\documentclass[a4paper, 11pt]{amsart}

\usepackage{amsmath,amsthm}
\usepackage{amssymb}

\usepackage[english]{babel}

\usepackage{color}

\newcommand{\mk}{\mathfrak}
\newcommand{\mc}{\mathcal}

\newcommand{\mf}{\mathbf}
\newcommand{\mb}{\mathbb}
\newcommand{\mr}{\mathrm}

\newtheorem{theo}{Theorem}

\newtheorem{lem}[theo]{Lemma}

\newtheorem{prop}[theo]{Proposition}

\newtheorem{rem}{Remark}

\newtheorem{conj}{Conjecture}

\newtheorem{cor}[theo]{Corollary}

\newtheorem{assumptions}[theo]{Assumptions}

\title{On the error of Fokker-Planck approximations of some one-step density dependent processes}

\author{D\'avid Kunszenti-Kov\'acs}
\address{MTA Alfr\'ed R\'enyi Institute of Mathematics, P.O. Box 127, H-1364 Budapest, Hungary}
\email{daku@renyi.hu}


\begin{document}

\maketitle

\begin{abstract}
Using operator semigroup methods, we show that Fokker-Planck type second-order PDE-s can be used to approximate the evolution of the distribution of a one-step process on $N$ particles governed by a large system of ODEs. The error bound is shown to be of order $O(1/N)$, surpassing earlier results that yielded this order for the error only for the expected value of the process, through mean-field approximations. We also present some conjectures showing that the methods used have the potential to yield even stronger bounds, up to $O(1/N^3)$.

\end{abstract}

\section{Introduction}

The aim of this paper is to expand the ideas presented in \cite{batkai2015pde} and \cite{KK-S2016}, where operator semigroup methods were used to provide a second-order PDE that can be used as a good approximation to a large Markovian stochastic system of two-state particles. Both papers provided a rough sketch of how to obtain bounds on the error of approximation. In this paper we give the full details of the proofs, highlighting where the difficulties lie. We also point towards further strengthening of the bounds and what they depend on.\\
One of the motivations is the study of disease propagation on Erd\H{o}s-R\'enyi type graphs, where, given any two subsets of their vertices, the number of edges between them is essentially proportional to the product of the sizes of the sets. This leads to a model where the dynamics are density dependent, that is governed by the number/proportion of vertices in the two possible states, and not by the specific location of these vertices.\\
Our stochastic model is as follows. We consider a system of $N$ identical elements or particles each of which can be in one of two states, denoted by $Q$ and $T$. The number of particles in state $Q$ at time $t$ is denoted by $X_Q(t)$ and similarly we use $X_T(t)$. These are considered to be random variables and our main goal is to derive PDEs yielding approximations for the distribution of these variables. The state of the system changes when the state of a particle changes, for which there are two possibilities: transitions $Q\to T$ and $T\to Q$. We shall assume that the transition rates depend solely on the proportion of nodes in each of the states, whereby the state of the system as a whole can be given by the number of particles of different types. The process is then a birth-death type process with a population cap.
The number of particles of each type changes by one in a short time interval, and the state of the system can be given by the pair $(k,N-k)$ yielding the number of particles of type $Q$ and $T$. In fact, the state will be given only by $k$, as the total number of particles does not change.
As usual, given the states of the system, we may define the transition rates, and derive the so-called master equations for the probability for each state. These contain all the information about the evolution of the sytem without any approximation, but the size of the family of equations increases with the number of particles. The mean-field ODE can be introduced to reduce the equation system to a single ODE. The stochastic convergence of the random variables to the solution of the mean-field equation has been widely studied by using martingale theory \cite{Bob:05, EthierKurtz, Kurtz1970}. Uniform convergence on bounded time intervals was proved in \cite{simon2013exact} by introducing an infinite system of ODEs for the moments. In \cite{batkai2012differential} this uniform convergence was also proved by using the approximation theory of operator semigroups and the authors showed that the difference is of order $1/N$. However, if we want to approximate the whole probability distributions, a PDE-based approach is needed. The approximation is based on a two-variable function $u$ for which $u(t, k/N) \approx p_k(t)$ and the Fokker-Planck equation. Then the master equations can be considered to be the discretisation of the Fokker-Planck equation in an appropriate sense. This was presented in \cite{batkai2015pde}, but since the main focus was a different type of second order PDE, very little detail was provided for the case we wish to study here, and some statements were inaccurately formulated. In \cite{KK-S2016}, an overview of the main building blocks of the proofs were given, but the details were left for this present paper. We therefore mostly follow the notations and proof structure introduced there.

In Section 2 the master equations and the corresponding Fokker-Planck equation are formulated, and we state our results. The  relation between the Fokker-Planck equation and the master equations is treated in Section 3, whilst the possible improvements are presented in Section 4.

\section{Model formulation}

Following the papers \cite{batkai2015pde, KK-S2016}, we will use systems with only two particle states $Q$ and $T$. Thus the state space of the whole system will be $\{ 0,1, \ldots , N\}$, where $k$ represents the state with $k$ particles in state $Q$ and $N-k$ particles in state $T$, that is $X_Q(t)=k$ and $X_T(t)=N-k$. Transition from state $k$ is possible only to states $k+1$ and $k-1$ with rates $a_k$ and $c_k$, respectively. Denoting by $p_k(t)$ the probability of state $k$ at time $t$ and assuming that the process is Markovian, the master equations of the process take the form
\begin{equation}
    \dot{p}_k= a_{k-1}p_{k-1}-(a_{k}+c_{k})p_k+c_{k+1}p_{k+1},\quad k=0,\ldots,  N . \label{eqKolm}
\end{equation}
Note that the equation corresponding to $k=0$ does not contain the first term in the right hand side, while the one corresponding to $k=N$ does not contain the third term, i.e., $a_{-1}=c_{N+1}=0$. Moreover, the Markov chain requires that $a_N$ and $c_0$ are set to zero. This will ensure that the sum of each column in the transition matrix is zero.

The infinite size limit, i.e. the case when $N\to \infty$, can be described by differential equations in 
The so-called density dependent case corresponds to when the transition rates $a_k$ and $c_k$ can be given by non-negative, continuous functions $A,C:[0,1]\to [0,+\infty)$ satisfying $A(1)=0=C(0)$ as follows
\begin{equation}
\frac{a_k}{N} = A\left( \frac{k}{N}\right)  \quad \mbox{ and } \quad \frac{c_k}{N} = C\left( \frac{k}{N}\right) . \label{dendep}
\end{equation}
We note that the conditions $A(1)=0=C(0)$ ensure $a_N=0=c_0$, whereas $a_1$ and $c_{N+1}$ are automatically "ignored" by restricting the correspondence to the interval $[0,1]$. Our aim here is estimate the error of the corresponding approximating Fokker-Planck equations \cite{risken2012fokker, van1992stochastic}.

The Fokker-Plank equation of the one-step-process given by density dependent coefficients is the following.
\begin{equation}
\partial_t u(t,z) = \frac{1}{2N} \partial_{zz} ((A(z) +C(z))u(t,z)) - \partial_{z} ((A(z)-C(z))u(t,z))
\label{FPdens}
\end{equation}
subject to boundary conditions
\begin{equation}\label{densbc1}
\delta \partial_z((A+C) u)(-\delta,t)-((A-C) u)(-\delta,t)=0,
\end{equation}
\begin{equation}\label{densbc2}
\delta \partial_z((A+C) u)(1+\delta,t)-((A-C) u)(1+\delta,t)=0,
\end{equation}
where $\delta = 1/2N$, and satisfies the initial condition
\begin{equation}
u(0,z) = u_0(z) \label{ICFP}
\end{equation}
for $z\in [0,1]$, where the initial function $u_0$ corresponds to the initial condition $p_k(0)$ in the sense that $u_0(\ell/N)=p_\ell(0)$ for all $0\leq \ell\leq N$.

For the derivation of this equation, we refer to \cite{KK-S2016}.

\subsection{Approximation results}

The solution $v$ of the Fokker-Plank equation was introduced as an approximation of the distribution $p_k$ in the sense that $v(t,k/N)\approx p_k(t)$. In \cite{batkai2015pde} Theorem 4.6 states that this is an order $1/N^2$ approximation on finite time intervals. However, the proof in fact yields a weaker result, as the definition of one of the norms absorbs a factor $1/N$, leading to a loss of one order, see Lemma \ref{le:norm_op} later in this paper. The correct statement is Corollary \ref{corFP} below.
\noindent We consider an initial function $u_0^N$ that is essentially the same for each $N$, and set
\begin{equation}
p_k(0):=\frac{1}{Q_N} u_0^N\left( \frac{k}{N} \right), \mbox{ with } Q_N=\sum_{k=0}^N u_0^N\left( \frac{k}{N} \right) \label{QN}
\end{equation}
as the initial condition for the ODE system, where the normalization constant ensures that $(p_0,p_1, \ldots , p_N)$ is a probability distribution. In order to estimate the accuracy of the Fokker-Planck equation precisely, we will need the following assumptions on the coefficient functions $A$ and $C$ and on the initial condition $u_0$.\\
The assumptions on the coefficient functions ensure that the obtained PDE is not degenerate and is compatible with the natural requirements of $a_N=c_0=0$ in the master equations.

\begin{assumptions}\label{ass:AC}
Let $A, C \in C^4[-\eta,1+\eta]$ with some $\eta>0$ satisfying $A+C>0$, $A(1)=C(0)=0$. Assume that $A-C$ is positive on $[-\eta,0]$ and negative on $[1,1+\eta]$. Moreover, let $N_0>1/2\eta$ be a positive integer such that
$2N_0|A(x)-C(x)|>|A'(x)+C'(x)|$ for all $x\in[-\eta,0]\cup[1,1+\eta]$.
\end{assumptions}

\noindent The assumptions on the initial functions reflect the wish to have a common, unique initial function when we restrict ourselves to the relevant domain $[0,1]$.

\begin{assumptions}\label{ass:u0}
Let $u_0\in C^2[0,1]$ be a non-negative function satisfying $u_0(0)=u_0'(0)=u_0(1)=u_0'(1)=0$, and let $u_0^N\in C^2[-h,1+h]$ be obtained as its extension as constant $0$ outside of $[0,1]$.
\end{assumptions}

\begin{theo} \label{thmFP}
Suppose that Assumptions \ref{ass:AC} and \ref{ass:u0} hold.
Let the coefficients of \eqref{eqKolm} be given by \eqref{dendep}, and let $q_k$ be the solution of \eqref{eqKolm} satisfying the initial conditions $q_{k}(0)=u_0(k/N)$ for all $k \in \{ 0,1,2,\ldots ,N\}$.
Let $u^N$ be the solution of the Fokker-Planck equation \eqref{FPdens} subject to the boundary condition \eqref{densbc1}-\eqref{densbc2} and initial condition $u^N(0,\cdot)=u_0^N(\cdot)$. Then for any $t_0>0$ there exists a constant $K$ independent of $N$, for which
$$\left|u^N\left(t,\frac{k}{N}\right) - q_k(t)\right|\leq K,\quad  t\in[0,t_0], \quad k=0,1,2,\ldots , N. $$
\end{theo}

The proof, which is based on a Trotter-Kato type result in the context of operator semigroups, is shown in Section \ref{sec:OpSemigroup}.

Note that $q_k$ is not a proper distribution since $\sum q_k \neq 1$, however, the above theorem translates easily to a statement concerning the distribution $p_k$ determined by \eqref{eqKolm}. Namely, $Q_N$, given in \eqref{QN}, relates $q_k$ to $p_k$ and $u^N$ to $v$ as follows. The functions $q_k$ and $p_k$ are solutions of the same system of linear differential equations \eqref{eqKolm}, and they are scalar multiples of each other. According to \eqref{QN} the relation $p_k=q_k/Q_N$ holds for all $k$. Similarly, $u^N$ and $v$ are solutions of the Fokker-Planck equation \eqref{FPdens} belonging to different initial conditions, hence they are related through $v(t,z)=u^N(t,z)/Q_N$. Moreover,
\[
Q_N=\sum_{\ell=0}^N u_0^N(\ell/N)=\sum_{\ell=0}^N u_0(\ell/N)=N\int_0^1 u_0 +o(N),
\]
hence we have the approximation of order $1/N$ for $p(t)$ and $v(t,z)$:
\[
\max_{0\leq k\leq N}\left|v(t,k/N)-p_k(t)\right| = \max_{0\leq k\leq N}\left|\frac{u^N(t,k/N)}{Q_N}-\frac{q_k(t)}{Q_N}\right|\leq \frac{K}{Q_N}\leq \frac{K'}{N}.
\]
Also, the boundary of the Fokker-Planck equation was chosen in such a way that the integral of the function $u^N$ is constant in time. Therefore, the approximation result given in Theorem \ref{thmFP} yields the following estimate between $p_k$ and $v$.

\begin{cor} \label{corFP}
Suppose that Assumptions \ref{ass:AC} and \ref{ass:u0} hold.
Let the coefficients of \eqref{eqKolm} be given by \eqref{dendep}, and let $p_k$ be the solution of \eqref{eqKolm} satisfying the initial conditions $p_{k}(0)=u_0(k/N)/Q_N$ for all $k \in \{ 0,1,2,\ldots ,N\}$, with $Q_N$ given in \eqref{QN}. Let $v$ be the solution of the Fokker-Planck equation \eqref{FPdens} subject to the boundary condition \eqref{densbc1}-\eqref{densbc2} and initial condition $v(0,\cdot)=u_0^N(\cdot)/Q_N$. Then for any $t_0>0$ there exists a constant $K$ independent of $N$, for which
$$\left|v\left(t,\frac{k}{N}\right) - p_k(t)\right|\leq\frac{K}{N},\quad  t\in[0,t_0], \quad k=0,1,2,\ldots , N. $$
\end{cor}

\section{Proof of the main result}  \label{sec:OpSemigroup}

The system of ODEs \eqref{eqKolm} can be written in the form $\dot p =A_N p$, where $p=(p_0,p_1,\ldots , p_N)^T$ and $A_N$ is a tri-diagonal matrix. The solution of the system can be given as $p(t)=T_N(t)p(0)$, where $T_N(t)=\exp(A_N t)$ is an operator semigroup on $\mathbb{R}^{N+1}$. (We note that it is extendable to $\mathbb{C}^{N+1}$ in a the usual way.) We will show that the solution of the Fokker-Planck equation \eqref{FPdens} can also be given by using an operator semigroup as $u(t,\cdot) = S_N(t) u(0,\cdot)$. Then we estimate the difference of the solutions by using a Trotter-Kato type result claiming that the semigroups are close to each other if this is known about their generators.

The generator of $S_N$ is given by the right hand side of the Fokker-Planck equation and will be denoted as
\begin{equation}
\mc{A}_N f = \frac{1}{2N}((A+C)f)''-((A-C)f)' . \label{opAN}
\end{equation}
Carrying out the differentiations and using the boundary conditions \eqref{densbc1}-\eqref{densbc2}, we get that the domain of this operator is the following subspace of the space of twice continuously differentiable functions
\begin{equation*}
D(\mc{A}_N):=\big\{f\in C^2[-h,1+h] \,: \frac{1}{2N}((A+C)f)'(z)-((A-C)f)(z)=0 \text{ for } z=-h, 1+h\big\} ,
\end{equation*}
where $h=-1/2N$. Now we introduce the general framework.

\subsection{Perturbation result in the abstract setting}

\begin{assumptions}\label{c:apro1.ass:approx_space}
Let $\mc{X}_n$, $X_n$ ($n\in\mb{N}^+$) be Banach spaces and assume that $P_n:\mc{X}_n\to X_n$ are bounded linear operators with $\|P_n\|\leq K$ for some constant $K>0$. Suppose that the operators $A_n$, $\mc{A}_n$ generate strongly continuous semigroups $(T_n(t))_{t\geq 0}$ and $(S_n(t))_{t\geq 0}$ on $X_n$ and $\mc{X}_n$, respectively, and that there are constants $M\geq 0$, $\omega\in\mathbb{R}$ such that the stability condition
\begin{equation}\label{c:apro1.eq:stability}
\|T_n(t)\|\leq Me^{\omega t} \qquad \text{ holds for all } t\geq 0.
\end{equation}
\end{assumptions}

Under these assumptions, we have the following Trotter-Kato type approximation result (cf., e.g., \cite[Proposition 3.8]{Batkai-Csomos-Farkas-Ostermann}).

\begin{prop}\label{prop:appr_first_gen}
Suppose that Assumptions \ref{c:apro1.ass:approx_space} hold, that there is a dense subset $Y_n\subset D(\mc{A}_n)$  invariant under the semigroup $S_n$ such that $P_nY_n\subset D(A_n)$, and that $Y_n$ is a Banach space with some norm $\|\cdot\|_{Y_n}$ satisfying
\begin{equation*}
\|S_n(t)\|_{Y_n} \leq M e^{\omega t}.
\end{equation*}
Let further $f\in Y_n$. If there exists a constant $p\geq 0$ with the property that for any $\tau\geq 0$ there exists a $C>0$ such that for all $\tau\geq t\geq 0$ the estimate
\begin{equation}\label{eqn:Gen}
\|A_nP_n S_n(t)f - P_n\mc{A}_nS_n(t)f\|_{X_n}\leq C\frac{\|S_n(t)f\|_{Y_n}}{n^p},
\end{equation}
holds, then for each $\tau\geq 0$ there exists some $C'>0$ such that
\begin{equation*}
\|T_n(t)P_n f - P_nS_n(t)f\|_{X_n}\leq C'\frac{\|f\|_{Y_n}}{n^p}
\end{equation*}
for all $0\leq t\leq\tau$, where $C'$ depends only on $C,\tau,M$ and $\omega$.
\end{prop}

The statement can be verified as follows. Let $f\in Y_n$, then the function $[0,t]\ni s\mapsto T_n(t-s)P_nS_n(s)f$ is continuously differentiable with derivative
\[
T_n(t-s)P_n\mc{A}_nS_n(s)f-T_n(t-s)A_nP_nS_n(s)f=T_n(t-s)(P_n\mc{A}_n-A_nP_n)S_n(s)f,
\]
and the fundamental theorem of calculus yields
\[
P_nS_n(t)f-T_n(t)P_nf=\int_0^t T_n(t-s)(P_n\mc{A}_n-A_nP_n)S_n(s) f \,\mathrm{d}s.
\]
Hence we have
\begin{align*}
\|T_n(t)P_n f - P_nS_n(t)f\|_{X_n}& \leq  \int_0^t \|T_n(t-s)(P_n\mc{A}_n-A_nP_n)S_n(s) f\|_{X_n}\, \mathrm{d}s\\
& \leq\int_0^t Me^{\omega (t-s)} \|(P_n\mc{A}_n-A_nP_n)S_n(s) f\|_{X_n}\,\mathrm{d}s\\
& \leq\int_0^t Me^{\omega (t-s)} C\frac{\|S_n(s) f\|_{Y_n}}{n^p}\,\mathrm{d}s
\leq
\int_0^t Me^{\omega (t-s)} C\frac{Me^{\omega s}\|f\|_{Y_n}}{n^p}\,\mathrm{d}s\\
& \leq C'\frac{\|f\|_{Y_n}}{n^p}
\end{align*}

with $C'=M^2C\tau e^{\tau|\omega|}$.

\subsection{Proof of Theorem \ref{thmFP}}

Now we turn to how this abstract setting applies to our case. Let Assumptions \ref{ass:AC} and \ref{ass:u0} hold. For each $N>N_0$, choose $X_N:= (\mathbb{C}^{N+1},\|\cdot\|_\infty)$ and $\mc{X}_N:=C[-h,1+h]$, with $P_N$ projecting $f\in \mc{X}_n$ onto the vector
\[
(f(0),f(1/N),\ldots,f(1))^T\in X_N.
\]
Clearly $\|P_N\|=1$. Let further $A_N$ be the transition matrix pertaining to the system of equations (\ref{eqKolm}) and $T_N(t)=\exp(A_N t)$. The operator $(\mc{A}_N,D(\mc{A}_N))$ given by \eqref{opAN} generates the analytic operator semigroup $(S_N(t))_{t\geq 0}$ on $\mc{X}_N$ that gives the solutions of PDE \eqref{FPdens} with boundary conditions \eqref{densbc1}-\eqref{densbc2}, cf. \cite[Section VI.4.b]{EN:00}. Since the semigroup is analytic, it leaves $D(\mc{A}_N)$ invariant. Thus using the notations of Theorem \ref{thmFP} we have
$$
q_k(t)= (T_n(t)P_n u^N(0,\cdot))_k, \mbox{ and } u^N(t,k/N)= (P_nS_n(t)u^N(0,\cdot))_k .
$$
Thus the statement of the Theorem follows directly from Proposition \ref{prop:appr_first_gen}. Now we formulate and prove a series of lemmas that will allow us to verify that the conditions of Proposition \ref{prop:appr_first_gen} hold.
The first set of lemmas are about the growth bounds of the semigroups in question, together with some of their restricitons.

\begin{lem}\label{le:exp_bound}
There exists a constant $\mf{d}>0$ such that for any $d>\mf{d}$ and $N>N_0$, the following hold:
\begin{enumerate}
\item for all $t\geq0$ the following norms are all bounded from above by $1$:
\begin{align*}
& \|e^{-dt}T_N(t)\|_{X_N};\mbox{ }
  \|e^{-dt}S_N(t)\|_{\mc{X}_N};\mbox{ }
\end{align*}
\item the space $Y_N:=\left(D(\mc{A}_N),\|\cdot\|_{\mc{A}_N-dI}\right)$ is a Banach space with the norm
\[
\|f\|_{\mc{A}_N-dI}:=\|(\mc{A}_N-dI)f\|_{\mc{X}_N},
\]
and for all $t\geq 0$ we have
\[
\left\|\left.e^{-dt}S_N(t)\right|_{Y_N}\right\|_{Y_N}\leq 1
\]
\end{enumerate}
\end{lem}

\begin{proof}
\underline{Part 1.}
First we shall show that for large enough $d_1$, we have
\[
 \|e^{-d_1t}T_N(t)\|_{X_N}\leq 1,
\]
independently of $N$. Note that this norm is simply the maximum norm on this finite dimesnional space. Let $v_0$ be an arbitrary initial vector, and let $v(t):=T_N(t)v(0)$. We have by definition
\begin{align*}
&\left((A_N-d_1I)v(t)\right)_k\\
=&NA\left(\frac{k-1}{N}\right)v_{k-1}(t)-
\left[
NA\left(\frac{k}{N}\right)+NC\left(\frac{k}{N}\right)+d_1
\right]v_k(t)
+NC\left(\frac{k+1}{N}\right)v_{k+1}(t)\\
=&v_k(t)
\left[
N\left(A\left(\frac{k}{N}\right)-A\left(\frac{k-1}{N}\right)\right)+
N\left(C\left(\frac{k+1}{N}\right)-C\left(\frac{k}{N}\right)\right)
-d_1
\right]\\
&+NA\left(\frac{k-1}{N}\right)\left(v_{k-1}(t)-v_k(t)\right)+NC\left(\frac{k+1}{N}\right)\left(v_{k+1}(t)-v_k(t)\right)
\end{align*}
for $1\leq k\leq N$. If $k=0$ or $k=N$, then either the terms with $A$ or the ones with $C$ become zero.
Fix $t> 0$, and suppose we have $|v_k(t)|=\|v(t)\|_{\infty}>0$, and consider the scalar product
\[
\left((A_N-d_1I)v(t)\right)_k\frac{\overline{v(t)_k}}{|v_k(t)|}.
\]
The following then holds, using that $A,C$ are both non-negative on $[0,1]$.
\begin{align*}
&\Re \left(\left((A_N-d_1I)v(t)\right)_k\frac{\overline{v(t)_k}}{|v_k(t)|}\right)\\
\leq&\frac{v_k(t)\overline{v(t)_k}}{|v_k(t)|}
\left[
N\left(A\left(\frac{k}{N}\right)-A\left(\frac{k-1}{N}\right)\right)\chi_{\{k>0\}}+
N\left(C\left(\frac{k+1}{N}\right)-C\left(\frac{k}{N}\right)\right)\chi_{\{k<N\}}
-d_1
\right]\\
=&|v_k(t)|
\left[
N\left(A\left(\frac{k}{N}\right)-A\left(\frac{k-1}{N}\right)\right)\chi_{\{k>0\}}+
N\left(C\left(\frac{k+1}{N}\right)-C\left(\frac{k}{N}\right)\right)\chi_{\{k<N\}}
-d_1
\right]\\
\leq&|v_k(t)|
\left[
\|A'\|_\infty+
\|C'\|_\infty
-d_1
\right]
\end{align*}
However, since $A_N-d_1I$ is the generator of the rescaled semigroup $e^{-d_1t}T_N(t)$, this means that if $d_1>\|A'\|_\infty+\|C'\|_\infty$, then the norm $\|(e^{-d_1t}T_N(t))v(0)\|_{X_N}=\|e^{-d_1t}v(t)\|_{\infty}$ is monotone decreasing for any $v(0)$, and we are done.

Next we turn to $\|e^{-d_2t}S_N(t)\|_{\mc{X}_N}$. It is known that $\mc{A}_N$ generates an analytic semigroup (cf. e.g. \cite[Thm. VI.4.6]{EN:00}), hence for any $v_0\in \mc{X}_N$ and $t>0$ we have $S_n(t)v_0=:v_t\in D(\mc{A}_N)$. Now assume that $v_0\neq 0$, fix $t> 0$ and suppose that $s\in [-h,1+h]$ is such that $|v_t(s)|=\|v_t\|_{\mc{X}_N}$. As above, our aim is to show that for large enough $d_2$ independent of $t$, the values of the derivative $(\mc{A}_N-d_2I)v_t$ and of the function $\overline{v}_t$ in $s$ have a negative scalar product (as vectors in $\mb{C}$), and hence
\[
\|(e^{-d_2t}S_N(t))v_0\|_{\mc{X}_N}=\|e^{-d_2t}v_t\|_{\infty}
\]
will be monotone decreasing in time. For ease of notation, we from now on write $w:=v_t$. We first show that $s$ has to be an interior point of the interval. Indeed, since $w\in D(\mc{A}_N)$, the function $w$ satisfies the boundary conditions, and we have at the left boundary that
\begin{align*}
0&=\frac{1}{2N}((A+C)w)'(-h)-((A-C)w)(-h)\\&=\frac{(A+C)(-h)}{2N}w'(-h)-\left[(A-C)(-h)-\frac{(A+C)'(-h)}{2N}\right]w(-h).
\end{align*}
Since $A+C$ is positive, we obtain
\[
w'(-h)=w(-h)\frac{2N(A-C)(-h)-(A+C)'(-h)}{(A+C)(-h)}.
\]
By the choice of $N_0$ and since $(A-C)(-h)>0$, we have for all $N>N_0$ that the coefficient of $w(-h)$ is positive, and hence $w'(-h)$ and $w(-h)$ have the same direction, meaning that $|w(-h)|\neq\|w\|_\infty$. Similar arguments can be applied at the other boundary, using $(A-C)(1+h)<0$.

We may thus assume that $s$ is an interior point. The scalar product we are interested in takes the form
\begin{align}
&(\mc{A}_Nw-d_2w)(s)\overline{w(s)}=\frac{(A+C)(s)}{2N}w''(s)\overline{w(s)}+
\left[
\frac{2(A+C)'(s)}{2N}-(A-C)(s)
\right]w'(s)\overline{w(s)}\nonumber\\\label{eqn:generator}+&
\left[
\frac{(A+C)''(s)}{2N}-(A-C)'(s)-d_2
\right]
w(s)\overline{w(s)}.
\end{align}
The condition $|w(s)|=\|w\|_\infty$ means that the function $z\mapsto \Re \left(w(z)\overline{w(s)}\right)$ takes its maximum in $z=s$, and so since $s$ is an interior point of the interval, we have
 $\Re \left(w'(s)\overline{w(s)}\right)=0$ and $\Re \left(w''(s)\overline{w(s)}\right)\leq 0$.
Using these after taking the real parts we obtain
\[
\Re\left((\mc{A}_Nw-d_2w)(s)\overline{w(s)}\right)\leq
\left[
\frac{(A+C)''(s)}{2N}-(A-C)'(s)-d_2
\right]
\|w\|_\infty^2,
\]
which is negative for any $d_2$ exceeding $d:=\|(A+C)''\|_\infty+\|(A-C)'\|_\infty$, a bound independent of $t$ and $N$. This proves the boundedness of the second family of norms.

\underline{Part 2.}
Writing $d_2=\alpha+\|(A+C)''\|_\infty+\|(A-C)'\|_\infty$ with some $\alpha>0$, the above inequality implies that we for any $w\in D(\mc{A}_N)$ have
\begin{equation}\label{eqn:norm_eq}
\|\mc{A}_Nw-d_2w\|_\infty>\alpha\|w\|_\infty,
\end{equation}
whereby $Y_N$ is indeed a Banach space with the desired norm. We showed that $\|e^{-d_2t}S_N(t)\|_{\mc{X}_N}$ is bounded from above by 1. But then for any $w\in Y_N$ we have
\begin{align*}
&\|e^{-d_2t}S_N(t)w\|_{Y_N}=\|e^{-d_2t}(\mc{A}_N-d_2)S_N(t)w\|_\infty=\|e^{-d_2t}S_N(t)(\mc{A}_N-d_2)w\|_\infty\\
=&\|e^{-d_2t}S_N(t)\left[(\mc{A}_N-d_2)w\right]\|_{\mc{X}_N}\leq \|(\mc{A}_N-d_2)w\|_{\mc{X}_N}=\|w\|_{Y_N}.
\end{align*}

To sum up, we may choose $\mf{d}:=\|(A+C)''\|_\infty+\|A'\|_\infty+\|C'\|_\infty$.
\end{proof}

Next, we want to bound the growth of the solutions in terms of the $C^1$ norm. To this end, first we fix $N>N_0$, and wish to find a suitable subspace of $C^1([-h,1+h])$ on which the semigroup is well-defined and strongly continuous also with respect to the new norm. We have seen that the semigroup is analytic on $\mc{X_N}$, and thus for any $v_0\in \mc{X}_N$, the elements $v_t$ of the orbit lie in $D(\mc{A}_N)$ for all $t>0$. In particular, each $v(t)$ satisfies the boundary conditions, hence if the initial condition $v_0$ does not, then continuity in the $C^1$ norm will automatically fail at $t=0$. Also this shows that any space containing $D(\mc{A}_N)$ remains invariant under the semigroup, which therefore may be restricted to any such subspace. Motivated by these observations, let
 \[
 Z_N:=\left\{f\in C^1([-h,1+h])\left|\frac{1}{2N}((A+C)f)'(z)-((A-C)f)(z)=0 \text{ for } z=-h, 1+h\right.\right\}
 \]
 be equipped with the standard $C^1$ norm.\\
The following lemma is rather technical, but its essence is that by applying a similarity transformation to $C([-h,1+h])$, we may eliminate the first order derivative term of the generator, allowing us to use inequality (\ref{eqn:norm_eq}) to compare the norm on $Y_N$ with the $C^2$ norm, and through this with the $C^1$ norm. Then we use the fact that we have a $C_0$-semigroup on $Y_N$, and take its unique extension to $Z_N$ to obtain the desired result.
 
\begin{lem}\label{le:semigroup}
For each $N>N_0$, the restriction of $S_N(\cdot)$ to $Z_N$ is a
$C_0$-semigroup with respect to the norm on $Z_N$.
\end{lem}
\begin{proof} Choose $d>\|(A+C)''\|_\infty+\|A'\|_\infty+\|C'\|_\infty$, and write
\[
f_1:=(A+C)/2N,\, f_2=C-A,\, g(z):=e^{\int_{-h}^{z}\frac{2f_1'+f_2}{2f_1}}.
\]
Then $g'=\frac{2f_1'+f_2}{2f_1}g$ and $g''=\left(\frac{2f_1'+f_2}{2f_1}\right)'g+\left(\frac{f_2}{2f_1}\right)^2g$.
Using this, for any $w\in D(\mc{A}_N)$ we have
\begin{align*}
&(\mc{A}_N-dI)w=f_1w''+(2f_1'+f_2)w'+(f_1''+f_2'-d)w\\
&=g^{-1}
\left[
f_1\left(gw''+\frac{2f_1'+f_2}{f_1}gw'+\left(\left(\frac{2f_1'+f_2}{f_1}\right)'+\left(\frac{f_2}{2f_1}\right)^2\right)gw\right)\right.\\
&+\left.\left\{
f_1''+f_2'
-f_1\left(
\left(\frac{2f_1'+f_2}{2f_1}\right)'+\left(\frac{f_2}{2f_1}\right)^2
\right)-d\right\}gw
\right]
.\\
&=g^{-1}
\left[
f_1(gw)''+\left\{
f_1''+f_2'
-f_1\left(
\left(\frac{2f_1'+f_2}{2f_1}\right)'+\left(\frac{f_2}{2f_1}\right)^2
\right)-d\right\}(gw)\right]\\
&=g^{-1}\left(f_1(gw)''+f_3(gw)\right),
\end{align*}
where
\[
f_3=f_1''+f_2'
-f_1\left(
\left(\frac{2f_1'+f_2}{2f_1}\right)'+\left(\frac{f_2}{2f_1}\right)^2\right)-d.
\]

We thus have for any $w\in D(\mc{A}_N)$, using inequality (\ref{eqn:norm_eq}) with the appropriate $\alpha>0$:
\begin{align*}
&\|f_1w''+f_3w\|_\infty\geq \|g^{-1}\left(f_1w''+f_3w\right)\|_\infty\cdot \|g^{-1}\|_\infty^{-1}
=\|g^{-1}\|_\infty^{-1}\cdot \|(\mc{A}_N-dI)w\|_\infty\\
&\geq \|g^{-1}\|_\infty^{-1}\cdot\alpha \|w\|_\infty\geq \left(\|f_3\|_\infty^{-1}\cdot\|g^{-1}\|_\infty^{-1}\cdot\alpha\right)\|f_3w\|_\infty.
\end{align*}
Clearly, if $d$ is large enough (recall that $N$ is fixed), $\|f_3\|_\infty\neq 0$. Let $\beta:=\|f_3\|_\infty^{-1}\cdot\|g^{-1}\|_\infty^{-1}\cdot\alpha$.  We then have $\|f_1w''+f_3w\|_\infty\geq\beta \|f_3w\|_\infty$, which by the properties of the maximum norm implies
\[
\|f_1w''+f_3w\|_\infty\geq\frac{\beta}{\beta+1}\|f_1w''\|_\infty.
\]
This yields
\begin{align*}
\|(\mc{A}_N-dI)w\|_\infty&=\|g^{-1}\left(f_1w''+f_3w\right)\|_\infty\geq \|g\|_\infty^{-1}\cdot
\|f_1w''+f_3w\|_\infty\\
\geq&\|g\|_\infty^{-1}\cdot\frac{\beta}{\beta+1}\|f_1w''\|_\infty\geq \frac{\beta}{\|g\|_\infty\|f_1^{-1}\|_\infty(\beta+1)}\|w''\|_\infty.
\end{align*}
Combined with inequality (\ref{eqn:norm_eq}), we obtain that the norm on $Y_N$ dominates the $C^2$ norm, which in turn dominates the $C^1$ norm.

 We have seen in Lemma \ref{le:exp_bound} that $S_N(\cdot)$ is a strongly continuous semigroup on $Y_N$ with respect to the $Y_N$-norm. But then this also holds true for the weaker norm $C^1$. Note that we know $S_N(\cdot)$ has a unique continuous extension to $\mc{X}_N$ with respect to the maximum norm, hence it also admits a unique strongly continuous closure with respect to the $C^1$ norm. But the closure of $D(\mc{A}_N)$ with respect to the $C^1$ norm is exactly $Z_N$.
\end{proof}
Now consider the subspace
\[
\mk{D}_N:=\left\{f\in D(\mc{A}_N)\left|\mc{A}_Nf\in Z_N\right.\right\}
\]
of $Z_N$. It can easily be seen that the generator of the $C_0$-semigroup $\left.S_N(t)\right|_{Z_N}$ is
\[
\mc{B}_N:=\left.\mc{A}_N\right|_{Z_N}.
\]
What remains to be shown is that we can provide exponential bounds for these restricted semigroups independently of $N>N_0$.
\begin{lem}\label{le:exp_bound2}
For any large enough constant $d>0$ we have for every $N>N_0$ that
\[
\left\|\left.e^{-dt}S_N(t)\right|_{Z_N}\right\|_{Z_N}\leq 1.
\]
\end{lem}

\begin{proof}
By classical PDE theory, since the functions $A$ and $C$ are both three times continuously differentiable, and the PDE is non-degenerate, the solutions $v_t:=S_N(t)v_0$ are all in $C^3([-h,1+h])$ for all $v_0\in \mc{X}_N$ and $t>0$. 
Also, we have for any $t>0$ that
\[
\mc{A}_N^2 v_t=\mc{A}_N S_N(t/2) v_{t/2}=\mc{A}_N S_N(t/2) \mc{A}_N v_{t/2},
\]
showing that $v_t\in D(\mc{A}_N^2)$ for all $t>0$.
However, we have 
\[
\mk{D}_N\supset C^3([-h,1+h])\cap D(\mc{A}_N^2),
\]
whence $v_t\in \mk{D}_N$ for each $v_0\in Z_n$ and $t>0$. From here, the arguments are very similar to those used in the proof of Lemma \ref{le:exp_bound}. We wish to show that for large enough $d>0$, the $Z_N$ norm along an orbit rescaled with $e^{-dt}$ is monotone decreasing for $t>0$.
To this end, let $v_0\in Z_N$, fix a $t>0$, and write $w:=S_N(t)v_0$. By the above, we have $w\in\mk{D}_N$. If $w=0$, the orbit stays constant $0$ for any larger time, rendering this case trivial. Assume that $w\neq 0$, and that $q,s\in[-h,1+h]$ are such that $\|w'\|=|w'(q)|$ and $\|w\|=|w(s)|$, and so
\[
\|w\|_{Z_N}=\|w'\|_\infty+\|w\|_\infty=|w'(q)|+|w(s)|.
\]
As seen in the proof of Lemma \ref{le:exp_bound}, $s$ is an interior point of the interval, however, $q$ need not be. It is enough to show that
\begin{align*}
F:=&\Re\left((\mc{B}_Nw-dw)'(q)\overline{\frac{w'(q)}{|w'(q)|}}\right)+
\Re\left((\mc{B}_Nw-dw)(s)\overline{\frac{w(s)}{|w(s)|}}\right)\\
=
&\Re\left((\mc{A}_Nw-dw)'(q)\overline{\frac{w'(q)}{|w'(q)|}}\right)+
\Re\left((\mc{A}_Nw-dw)(s)\overline{\frac{w(s)}{|w(s)|}}\right)
\end{align*}
is negative for any large enough $d$ independent of $w\in\mk{D}_N$, $t>0$ and $N>N_0$.
We have that
\begin{align*}
&(\mc{A}_Nw-dw)'(q)=
\frac{(A+C)(q)}{2N}w'''(q)+
\left[
\frac{3(A+C)'(q)}{2N}-(A-C)(q)
\right]w''(q)\\+&
\left[
\frac{3(A+C)''(q)}{2N}-2(A-C)'(q)-d
\right]
w'(q)
+\left[
\frac{(A+C)'''(q)}{2N}-(A-C)''(q)
\right]
w(q).
\end{align*}
 First let us consider the case when $q$ is an interior point of $[-h,1+h]$. Then by the condition on $q$ we obtain
 $\Re \left(w''(q)\overline{w'(q)}\right)=0$ and $\Re \left(w'''(q)\overline{w'(q)}\right)\leq 0$. Using that $A+C>0$ and combining with the expressions for $s$, we have
\begin{align*}
&F=\Re\left((\mc{A}_Nw-dw)'(q)\overline{\frac{w'(q)}{|w'(q)|}}\right)+
\Re\left((\mc{A}_Nw-dw)(s)\overline{\frac{w(s)}{|w(s)|}}\right)\\&\leq
\left[
\frac{3(A+C)''(q)}{2N}-2(A-C)'(q)-d
\right]
\|w'\|_\infty
+\left[
\frac{(A+C)'''(q)}{2N}-(A-C)''(q)
\right]
w(q)\overline{\frac{w'(q)}{|w'(q)|}}
\\&+
\left[
\frac{(A+C)''(s)}{2N}-(A-C)'(s)-d
\right]
\|w\|_\infty\leq
\left[
\frac{3(A+C)''(q)}{2N}-2(A-C)'(q)-d
\right]
\|w'\|_\infty\\&+
\left[
\frac{(A+C)''(s)}{2N}-(A-C)'(s)+\left|
\frac{(A+C)'''(q)}{2N}-(A-C)''(q)
\right|-d
\right]
\|w\|_\infty,
\end{align*}
which is indeed negative for any large enough $d$ not depending on $t>0,N>N_0,w\in \mk{D}_N$, provided $w\neq0$.

Now we have to deal with the case when $q$ is allowed to be one of the endpoints of the interval. By symmetry, it is enough to consider $q=-h$, in which case
\begin{equation}\label{eqn:deriv_ineq}
\Re \left(w''(-h)\overline{w'(-h)}\right)\leq 0.
\end{equation}
Note that if this holds with equality, then we must have $\Re \left(w'''(-h)\overline{w'(-h)}\right)\leq 0$, and the above arguments work. Therefore we may assume $\Re \left(w''(-h)\overline{w'(-h)}\right)< 0$.
Recall that $S_N(\tau)\in D(\mc{A}_N)$ for all $\tau>0$, hence the functions $v_\tau$ all satisfy the boundary condition at $-h$, and so the functions $e^{-d\tau}v_\tau'(-h)$ and $e^{-d\tau}v_\tau(-h)$ only differ by a positive constant factor $\gamma:=\frac{2N(A-C)(-h)-(A+C)'(-h)}{(A+C)(-h)}$. We thus obtain
\begin{align*}
&\Re\left((\mc{A}_Nw-dw)'(-h)\overline{\frac{w'(-h)}{|w'(-h)|}}\right)=
e^{dt}\Re\left(\partial_z\left.\left(\partial_\tau(e^{-d\tau}v_\tau(z))|_{\tau=t}\right)\right|_{z=-h}\cdot\overline{\frac{w'(-h)}{|w'(-h)|}}\right)\\
=&e^{dt}\Re\left(\partial_\tau\left.\left(\partial_z(e^{-d\tau}v_\tau(z))|_{z=-h}\right)\right|_{\tau=t}\cdot\overline{\frac{w'(-h)}{|w'(-h)|}}\right)\\
=&e^{dt}\Re\left(\partial_\tau\left.\left(\partial_z(e^{-d\tau}v_\tau(z))|_{z=-h}\right)\right|_{\tau=t}\cdot\overline{\frac{w'(-h)}{|w'(-h)|}}\right)=
e^{dt}\Re\left(\partial_\tau\left.\left(\gamma e^{-d\tau}v_\tau(-h)\right)\right|_{\tau=t}\cdot\overline{\frac{w(-h)}{|w(-h)|}}\right)\\
=&\gamma\Re\left((\mc{A}_Nw-dw)(-h)\overline{\frac{w(-h)}{|w(-h)|}}\right).
\end{align*}
Also, we have
\[
\gamma\Re \left(w''(-h)\overline{w(-h)}\right)=\Re \left(w''(-h)\overline{w'(-h)}\right)\leq 0,
\]
and combining this with equation (\ref{eqn:generator}) yields
\begin{align*}
&\Re\left((\mc{A}_Nw-dw)'(-h)\overline{\frac{w'(-h)}{|w'(-h)|}}\right)=
\gamma\frac{(A+C)(-h)}{2N}\Re\left(w''(-h)\overline{\frac{w(-h)}{|w(-h)|}}\right)\\+&
\gamma\left[
\frac{2(A+C)'(-h)}{2N}-(A-C)(-h)
\right]\gamma|w(-h)|\\
+&\gamma
\left[
\frac{(A+C)''(-h)}{2N}-(A-C)'(-h)-d
\right]
|w(-h)|\\
\leq&
|w'(-h)|\left[
\frac{2(A+C)'(-h)}{2N}-(A-C)(-h)+\frac{(A+C)''(-h)}{2N}-(A-C)'(-h)-d
\right]
,
\end{align*}
which is negative for any large enough $d>0$, independently of $t>0$, $N>N_0$ and $0\neq w\in\mk{D}_N$. The term with $s$ also has this property by the proof of Lemma \ref{le:exp_bound}, completing our argument.
\end{proof}

\noindent Next, we shall turn our attention to the error bound between the generators, and fix $N>N_0$. The second degree Taylor expansion with Lagrange remainder term will be used to estimate the left hand side of \eqref{eqn:Gen}:
$$
\max \{ (A_N P_N f - P_N \mc{A}_N f )_k\ : k=0,1, \ldots , N \} .
$$
First however let us define two functions $F_1,F_2: [-2h,1+2h]\to\mb{R}$ in the following way. We let $F_1(z)=((A+C)f)(z)$ and $F_2=((A-C)f)(z)$ for $-h\leq z\leq 1+h$, and then extend both functions to the full interval so as to keep them in $C^2([-2h,1+2h])$, whilst having $\|F_1\|_{C^2}\leq 4\|(A+C)f\|_{C^2}$ and $\|F_2\|_{C^2}\leq 4\|(A-C)f\|_{C^2}$.

For $1\leq k\leq N-1$, using the tridiagonal form of the matrix $A_N$, the first term can be written as
$$
(A_N P_N f)_k = a_{k-1}f\left(\frac{k-1}{N}\right)-(a_{k}+c_{k})f\left(\frac{k}{N}\right)+c_{k+1}f\left(\frac{k+1}{N}\right) ,
$$
and exploiting the density dependence \eqref{dendep}, leads to
$$
(A_N P_N f)_k = N\left( (Af)\left(\frac{k-1}{N}\right)-((A+C)f)\left(\frac{k}{N}\right)+(Cf)\left(\frac{k+1}{N}\right)   \right) .
$$
This can be artificially rearranged to
\begin{align*}
(A_N P_N f)_k &= \frac{N}{2} \left( F_1\left(\frac{k-1}{N}\right)-2F_1\left(\frac{k}{N}\right)+F_1\left(\frac{k+1}{N}\right) \right) \\
&+ \frac{N}{2} \left(F_2\left(\frac{k-1}{N}\right) - F_2\left(\frac{k+1}{N}\right) \right) .
\end{align*}
The second degree Taylor formula with Lagrangian remainder will be used in the form
$$
F(z+\eta)= F(z) + F'(z)\eta + F''(z+\zeta)\frac{\eta^2}{2} ,
$$
where $\zeta$ is between $0$ and $\eta$. This will be applied with the choices $z=k/N$, $\eta= 1/N$, $\eta= -1/N$, $F=F_1$ and $F=F_2$ leading to
\begin{align*}
(A_N P_N f)_k &= \frac{N}{2} \left( \frac{1}{2N^2}F_1''\left(\frac{k}{N} -\zeta_1\right) + \frac{1}{2N^2}F_1''\left(\frac{k}{N} +\zeta_2\right) \right)\\ &-\frac{N}{2} \left(\frac{2}{N}F_2'\left(\frac{k}{N}\right) - \frac{1}{2N^2}F_2''\left(\frac{k}{N} -\zeta_3\right) + \frac{1}{2N^2}F_2''\left(\frac{k}{N} +\zeta_4\right)  \right)
\end{align*}
with $\zeta_i$ between zero and $1/N$.

Now using \eqref{opAN} we have
$$
(P_N \mc{A}_N f )_k= \frac{1}{2N}F_1''\left(\frac{k}{N}\right)-F_2'\left(\frac{k}{N}\right) ,
$$
hence the difference of the two generators can be expressed as
\begin{align*}
(A_N P_N f)_k - (P_N \mc{A}_N f )_k &= \frac{1}{4N} \left( F_1''\left(\frac{k}{N} -\zeta_1\right) - F_1''\left(\frac{k}{N}\right) \right)\\
&+ \frac{1}{4N} \left( F_1''\left(\frac{k}{N} +\zeta_2\right) - F_1''\left(\frac{k}{N}\right) \right)\\
&+ \frac{1}{4N} \left( F_2''\left(\frac{k}{N} -\zeta_3\right) - F_2''\left(\frac{k}{N} +\zeta_4\right)   \right) .
\end{align*}
The difference can thus be estimated as follows:
$$
|(A_N P_N f)_k - (P_N \mc{A}_N f )_k| \leq \frac{1}{2N} (2\|F_1\|_{C^2}+ \|F_2\|_{C^2} )
\leq\frac{2}{N} (2\|(A+C)f\|_{C^2}+ \|(A-C)f\|_{C^2} ) .
$$

Of the cases $k=0, k=N$, we shall only detail the former, the latter follows in a similar manner.
We have on the one hand
\begin{align*}
(A_N P_N f)_0 &= -a_{0}f\left(\frac{0}{N}\right)+c_{1}f\left(\frac{1}{N}\right)=N\left(-(Af)\left(0\right)+(Cf)\left(\frac{1}{N}\right)   \right)\\
&=N\left( (Af)\left(\frac{-1}{N}\right)-((A+C)f)\left(0\right)+(Cf)\left(\frac{1}{N}\right)   \right)\\
&+\frac{N}{2}\left(
 ((A+C)f)\left(\frac{-1}{N}\right)+((A-C)f)\left(\frac{-1}{N}\right)
\right)
\\
&=\frac{N}{2} \left( F_1\left(\frac{-1}{N}\right)-2F_1\left(\frac{0}{N}\right)+F_1\left(\frac{1}{N}\right) \right) \\
&+ \frac{N}{2} \left(F_2\left(\frac{-1}{N}\right) - F_2\left(\frac{1}{N}\right) \right)+
\frac{N}{2}\left(
 F_1\left(\frac{-1}{N}\right)+F_2\left(\frac{-1}{N}\right)
 \right),
\end{align*}
where the first two terms are identical to those obtained for $1\leq k\leq N-1$, hence we only need to estimate the last term.
On the other hand, since we assume $f\in Y_N$, we can exploit the boundary condition at $-h$, that is:
\[
F_1'(-h)-2N F_2(-h)=0.
\]

\noindent Using the second degree Taylor formula with Lagrangian remainder as above we obtain
\begin{align*}
F_1(0)-F_1(-2h)-F_2(0)-F_2(-2h)&=2hF_1'(-h)+\frac{h^2}{2}\left(F_1''(-h+\xi_1)-F_1''(-h-\xi_2)\right)\\
&-2F_2(-h)-\frac{h^2}{2}\left(F_2''(-h+\xi_3)+F_2''(-h-\xi_4)\right),
\end{align*}
where $\xi_j\in[0,h]$, which after rearranging and taking into account the boundary condition, $2Nh=1$ and $C(0)=0$ yields
\begin{align*}
& F_1\left(\frac{-1}{N}\right)+F_2\left(\frac{-1}{N}\right)=F_1(0)-F_2(0)+\frac{1}{N}(F_1'(-h)-2NF_2(-h))\\
&+\frac{h^2}{2}(F_1''(-h+\xi_1)-F_1''(-h-\xi_2)-F_2''(-h+\xi_3)-F_2''(-h-\xi_4))\\
&= \frac{h^2}{2}(F_1''(-h+\xi_1)-F_1''(-h-\xi_2)-F_2''(-h+\xi_3)-F_2''(-h-\xi_4)).
\end{align*}
That is, compared to the cases $1\leq k\leq N-1$ we have an extra error term for $k=0$ and $k=N$ that can be bounded by
\[
\frac{Nh^2}{2}\left(\|F_1\|_{C^2}+ \|F_2\|_{C^2}\right)=\frac{1}{8N}\left(\|F_1\|_{C^2}+ \|F_2\|_{C^2}\right)
\leq \frac{1}{2N} (\|(A+C)f\|_{C^2}+ \|(A-C)f\|_{C^2} )
\]
Thus we proved the following statement.

\begin{lem}\label{le:num_bound}
There exists a constant $C_0$ such that for any $N\geq N_0$ and $g\in D(\mc{A}_N)$ we have
\[
\max_{0\leq k\leq N}|(A_NP_ng)_k -(P_N\mathcal{A}_Ng)_k|\leq \frac{C_0}{N}(\|(A+C) g\|_{C^2}+\|(A-C) g\|_{C^2})
\]
\end{lem}

\begin{rem} In light of the above estimates, the choice of the coefficient of the second order term in $\mc{A}_N$ may seem arbitrary, as similar estimates hold with any other function as long as it is kept of magnitude $O(1/N)$. However, if we were to use the $C^3$ norm, then this is the choice that yields an estimate of order $1/N^2$ for interior points. The reason for not using this estimate here is that the discretization of the boundary condition does not lead to an estimate better than $O(1/N)$ anyway. For possible improvements we refer to Section \ref{sect:outlook}.
\end{rem}

To prove Proposition \ref{prop:appr_first_gen}, we need that with $g=S_N(t)f$, the upper bound on the right hand side can be dominated by $C_1\|g\|_{Y_N}$ for all $0\leq t\leq \tau$ and $N>N_0$.
This shall be done using Lemmas \ref{le:exp_bound} and \ref{le:exp_bound2}.

\begin{lem}\label{le:norm_op}
For any $\tau\geq 0$ there exists a constant $C_1$ such that for any $N\geq N_0$ we have
\[
\|(A+C) u^N(t,\cdot)\|_{C^2}+\|(A-C) u^N(t,\cdot)\|_{C^2}\leq C_1 N\|u_0^N\|_{\mc{A}_N-dI}
\]
for all $0\leq t\leq \tau$.
\end{lem}
\begin{proof}
We have the following series of inequalities.
\begin{align*}
&\|(A+C) u^N(t,\cdot)\|_{C^2}+\|(A-C) u^N(t,\cdot)\|_{C^2}\leq \left(1+\left|\frac{A-C}{A+C}\right\|_{C^2}\right)\|(A+C) u^N(t,\cdot)\|_{C^2},
\end{align*}
and
\begin{align*}
&\|(A+C) u^N(t,\cdot)\|_{C^2}\leq\|(A+C) u^N(t,\cdot)\|_{C^1}+\|\left((A+C) u^N(t,\cdot)\right)''\|_{\infty}\\
\leq&
\|(A+C) u^N(t,\cdot)\|_{C^1}+\|\left((A+C) u^N(t,\cdot)\right)''+2N\left((A-C)u^N(t,\cdot)\right)'-2dNu^N(t,\cdot)\|_{\infty}\\
&+2N\|\left((A-C)u^N(t,\cdot)\right)'\|_{\infty}+2dN\|u^N(t,\cdot)\|_\infty\\
\leq& N\|u^N(t,\cdot)\|_{Y_N}+\|u^N(t,\cdot)\|_{Z^N}
\left(
\|A+C\|_{C^1}+2N\|A-C\|_{C^1}+2dN
\right)\\
\leq&
e^{d\tau}N\|u_0^N\|_{Y_N}+e^{d\tau}\|u_0^N\|_{C^1}
\left(
\|A+C\|_{C^1}+2N\|A-C\|_{C^1}+2dN
\right)
\end{align*}
\end{proof}
\noindent Notice that
\begin{align*}
\|u_0^N\|_{Y_N}=\|(\mc{A}_N-dI)u_0^N\|&\leq \|\partial_{zz}((A+C)u_0^N)\|+\|\partial_{z}((A-C)u_0^N)\|
+d\|u_0^N\|\\
&=\|\partial_{zz}((A+C)u_0)\|+\|\partial_{z}((A-C)u_0)\|
+d\|u_0\|= const.,
\end{align*}
whereas
\begin{align*}
\|u^N(t,\cdot)\|_{Y_N}&=\|(\mc{A}_N-dI)u^N(t,\cdot)\|\geq\frac{1}{1+2h}\int_{-h}^{1+h}du^N(t,z)+\left(\mc{A}_Nu^N(t,\cdot)\right)(z) \,\mr{dz}\\
=&\frac{1}{1+2h}\left(
d\int_{-h}^{1+h}u^N(t,z)\,\mr{dz}
+\partial_t
\int_{-h}^{1+h}u^N(t,z)\,\mr{dz}
\right)=\frac{d}{1+2h}\int_{-h}^{1+h}u_0^N(z)\,\mr{dz}\\
\geq& const.,
\end{align*}
since the boundary conditions guarantee that the integral of the solution remains constant.
Putting these observations together, we have proved that \eqref{eqn:Gen} holds. Hence all the conditions of Proposition \ref{prop:appr_first_gen} are fulfilled with $p=0$, and that finishes the proof of Theorem~\ref{thmFP}.

Finally, further exploiting the analyticity of the semigroups involved, this result can be extended to a larger space of initial conditions $f$, weakening the regularity assumptions needed for such an approximation to hold.
Suppose $\mc{A}$ generates an analytic semigroup $S(t)$ with $\|S(t)\|\leq Me^{\omega t}$.
Let $\varepsilon\in (0,1)$, $q> \omega$, and define 
$Y:=D((qI-\mc{A}))$ and $Z:= D((qI-\mc{A})^{\varepsilon})$ with the corresponding operator induced norms.
It is known that we then have (cf. \cite[Corollary 9.22.c]{Batkai-Csomos-Farkas-Ostermann})
\begin{equation*}
\|S(s)\|_{\mathcal{L}(Z,Y)}\leq \frac{e^{qs}M}{s^{1-\varepsilon}}.
\end{equation*}

Thus we obtain the following result (with $p=0$ and $q=d_1>d$ in our case).

\begin{lem}\label{lem:anal_approx}
Suppose that the conditions of Proposition \ref{prop:appr_first_gen} are satisfied and that $\mc{A}_N$ generates an analytic semigroup $S_N(t)$. For the above choice of $Y$ and $Z$ we have
\begin{equation*}
\|T_N(t)P_N f - P_NS_N(t)f\|_{X_N}\leq C''\frac{\|f\|_Z}{N^p}
\end{equation*}
for all
$f\in Z$.
\end{lem}
\begin{proof}
Indeed, for an analytic semigroup $S_N(\cdot)$, the map $[0,t]\ni s\mapsto T_N(t-s)P_NS_N(s)f$ is continuous and continuously differentiable on $(0,t]$ for any $f\in X$, and in particular for $f\in Z$. The proof of Proposition \ref{prop:appr_first_gen} can then be followed to obtain
\begin{align*}
\|T_N(t)P_N f - P_NS_N(t)f\|_{X_N}&\leq\int_0^t Me^{\omega (t-s)} C\frac{\|S_N(s) f\|_{Y}}{N^p}\,\mathrm{d}s
\leq
\int_0^t Me^{qt} C\frac{M\|f\|_{Z}}{N^ps^{1-\varepsilon}}\,\mathrm{d}s\\
& \leq C''\frac{\|f\|_Z}{N^p},
\end{align*}
where we used the finiteness of the integral $\int_0^t \frac{1}{s^{1-\varepsilon}}\,\mathrm{d}s$ whenever $\varepsilon\in(0,1)$.
\end{proof}

\section{Some open questions}\label{sect:outlook}

In this last section, we consider some open questions, and sketch possible approaches to improving the bounds achieved, specifically, to increase the value of the exponent of $N$ to 2, or even 3 in the bound of Corollary \ref{corFP}.

Looking at the coefficients in the generators $\mc{A}_N$, the first order terms induce a drift from the intervals $[-h,0]$ and $[1,1+h]$ towards the middle interval, and the relative effect of the diffusion from the second order term decreases as $N$ increases. This motivates the following conjectures:
\begin{conj}\label{conj:C2}
Consider an initial function $u_0$ satisfying the conditions given in Theorem~\ref{thmFP}. Then for any $t>0$ and $N>N_0$ we have
\begin{equation*}
|(S_N(t)u_0)''(-h)|<\|(S_N(t)u_0)''\|_\infty \mbox{ and } |(S_N(t)u_0)''(1+h)|<\|(S_N(t)u_0)''\|_\infty.
\end{equation*}
\end{conj}

\begin{conj}\label{conj:C3}
Consider an initial function $u_0\in C^3([0,1])$ with $u_0'''(0)=u_0'''(1)$ satisfying the conditions given in Theorem~\ref{thmFP}. Then for any $t>0$ and $N>N_0$ we have
\begin{equation*}
|(S_N(t)u_0)'''(-h)|<\|(S_N(t)u_0)'''\|_\infty \mbox{ and } |(S_N(t)u_0)'''(1+h)|<\|(S_N(t)u_0)'''\|_\infty.
\end{equation*}
\end{conj}

In essence, these corollaries state that the second and third derivatives cannot achieve their maximum at the endpoints if we are on an orbit starting from a non-negative function with support smaller than the whole interval. These conjectures would lead to the following results.

\begin{prop}\label{prop:C2}
If Conjecture \ref{conj:C2} holds, then for any large enough constant $d>0$ we have for every $N>N_0$ that
\[
\left\|\left.e^{-dt}S_N(t)\right|_{Z_{2,N}}\right\|_{Z_{2,N}}\leq 1,
\]
where $Z_{2,N}:=D(\mc{A}_N)$ is equipped with the $C^2$ norm.
\end{prop}
\begin{proof}
We have seen in the proof of Lemma \ref{le:semigroup} that the restriction to $Z_{2,N}$ is still a $C_0$-semigroup due to the equivalence of norms.
It is enough to follow the proof of Lemma \ref{le:exp_bound2}, and note that we by our conjecture do not need to worry about the case when the relevant derivatives take their maximum at the boundaries.
\end{proof}

\begin{prop}\label{prop:C3}
If Conjectures \ref{conj:C2} and \ref{conj:C3} hold, and the semigroups restricted to appropriate subspaces $Z_{3,N}\subset C^3([-h,1+h])$ are strongly continuous, then for any large enough constant $d>0$ we have for every $N>N_0$ that
\[
\left\|\left.e^{-dt}S_N(t)\right|_{Z_{3,N}}\right\|_{Z_{3,N}}\leq 1.
\]
\end{prop}

A final - less evident - conjecture asserts that not only do the second derivatives not take their largest value at the boundaries, but that they actually stay very small as a function of $N$ in a whole vicinity of the boundaries.

\begin{conj}\label{conj:C2_bound}
Consider an initial function $u_0$ satisfying the conditions given in Theorem~\ref{thmFP}. Then for any $\tau>0$ there exists a constant $K'$ such that for any $t\in[0,\tau]$ and $N>N_0$ we have
\[
\max_{z\in[-h,0]\cup[1,1+h]}|(S_N(t)u_0)''(z)|\leq K'/N.
\]
\end{conj}

\begin{theo}
Assuming Conjecture \ref{conj:C2}, the estimate in Corollary \ref{corFP} takes the form
$$\left|v\left(t,\frac{k}{N}\right) - p_k(t)\right|\leq\frac{K}{N^2},\quad  t\in[0,t_0], \quad k=0,1,2,\ldots , N. $$
\end{theo}
\begin{proof}
Note that we may gain a factor $N$ by applying Proposition \ref{prop:C2} right after the first inequality in the proof of Lemma \ref{le:norm_op}.
\end{proof}

\begin{theo}
Assuming Conjectures \ref{conj:C2}, \ref{conj:C3} and \ref{conj:C2_bound} the estimate in Corollary \ref{corFP} takes the form
$$\left|v\left(t,\frac{k}{N}\right) - p_k(t)\right|\leq\frac{K}{N^3},\quad  t\in[0,t_0], \quad k=0,1,2,\ldots , N. $$
\end{theo}
\begin{proof}
Note that Conjecture \ref{conj:C2_bound} improves the estimate on the boundaries by one order. This in turn makes it reasonable to apply Taylor approximations with third order Lagrangian remainder term for internal points, leading to error bounds of order $1/N^2$ in the $C^3$ norm in Lemma \ref{le:num_bound}.\\
Conjectures \ref{conj:C2}, \ref{conj:C3} on the other hand yield a version of Lemma \ref{le:norm_op} for the $C^3$ norms with no factor $N$ on the right hand side (cf. the argument in the previous proof), keeping the order $1/N^2$ bound. Rescaling the function $u$ to the function $v$ yields the last required factor $1/N$.
\end{proof}

\section*{Acknowledgements}
The author has received funding from the European Research Council under the European Union's Seventh Framework Programme (FP7/2007-2013) / ERC grant agreement $\mathrm{n}^\circ$617747, and from the MTA R\'enyi Institute Lend\"ulet Limits of Structures Research Group.


\end{document}